\documentclass[a4paper]{article}
\usepackage{amsrefs,amssymb,amsmath,amsthm,amsfonts,graphicx,epsfig}
\usepackage[dvips]{color}
\usepackage[utf8]{inputenc}
\usepackage[english]{babel}

\theoremstyle{plain}
\newtheorem{teo}{Theorem}[section]

\newtheorem{lemma}[teo]{Lemma}
\newtheorem{prop}[teo]{Proposition}

\newtheorem{obs}[teo]{Remark}

\theoremstyle{definition}
\newtheorem{df}[teo]{Definition}

\theoremstyle{remark}

\DeclareMathOperator{\dist}{dist} 
\DeclareMathOperator{\tdist}{\tilde\dist} 

\DeclareMathOperator{\orb}{orb}
\newcommand{\reparam}{\mathcal{H}^+}
\newcommand{\reparamx}{\mathcal{H}}
\newcommand{\rep}{\mathcal{H}}

\newcommand{\R}    {\mathbb R}

\newcommand{\K}    {\mathbb K}

\newcommand{\N}  {\mathbb N}

\renewcommand{\epsilon}{\varepsilon}

\renewcommand{\phi}{\Phi}
\begin{document}

\author{Alfonso Artigue
\footnote{Email: artigue@unorte.edu.uy. Adress: Departamento de Matemática y Estadística del Litoral, Universidad de la Rep\'ublica, Gral. Rivera 1530, Salto, Uruguay.}
}

\title{Positive expansive flows}
\date{\today}
\maketitle
\begin{abstract}
We show that every positive expansive flow on a compact 
metric space consists of a finite number of periodic orbits and fixed points. 
\end{abstract}
\section*{Introduction}

In the subject of discrete dynamical systems, a homeomorphisms $f\colon X\to X$ on a compact metric space $X$ is said to be 
\emph{positive expansive} if there is $\alpha>0$ such that if $\dist(f^n(x),f^n(y))<\alpha$ for all $n\geq 0$ then $x=y$. 
It is well known that if $X$ admits a positive expansive homeomorphisms then $X$ is finite, see for example \cites{Lew,CK}. 
Here we will show the corresponding result for positive expansive flows. 
We will consider the definition of R. Bowen and P. Walters \cite{BW} for expansive 
flows without singularities (see Definition \ref{BWexp}) and 
the definition of M. Komuro \cite{K} for flows with singular points (see Definition \ref{komuroexp}). 
In both cases we show that every positive expansive flow has only a finite number of orbits being each one compact, 
i.e. periodic or singular (Theorems \ref{teo1} and \ref{teo2}).

The proofs known to the author, in the discrete case, start showing that every point is Lyapunov stable for $f^{-1}$, that is, for all 
$x\in X$ and $\epsilon>0$ there is $\delta>0$ such that 
if $\dist(x,y)<\delta$ then $\dist(f^{-n}(x),f^{-n}(y))<\epsilon$ for all $n\geq 0$. 
Let us recall how this is proved in \cite{Lew}. 
By contradiction suppose that $x\in X$ is not stable for $f^{-1}$. 
So there is $\epsilon\in (0,\alpha)$ and a sequence $y_j\to x$ as $j\to\infty$ 
such that for all $j\in\N$ there is $n_j\in \N$, $n_j\to\infty$, with the property 
$$\dist(f^{-n_j-1}(y_j),f^{-n_j-1}(x))\geq\epsilon$$
and 
$$\dist(f^{-n}(y_j),f^{-n}(x))<\epsilon$$
for all $n=0,1,\dots, n_j$.
Since $f$ is continuous and $X$ is compact there is $\sigma>0$ such that $\dist(f^{-n_j}(y_j),f^{-n_j}(x))\geq\sigma$ for all $j\in\N$. 
Assuming that $f^{-n_j}(y_j)\to y_ *$ and $f^{-n_j}(x)\to x_ *$ we have that $\dist(y_*,x_*)\geq\sigma$ and $x_*\neq y_*$.
And by continuity we have that 
$$\dist(f^n(x_*),f^n(y_*))\leq\epsilon$$ 
for all $n\geq 0$, contradicting the positive expansiveness of $f$.

In the continuous case we consider positive expansive flows allowing reparameterizations, see Definition \ref{BWexp}. 
So, as in \cite{Ma}, we consider the concept of Lyapunov stability allowing reparameterizations, see Definition \ref{dfstable}. 
Following the ideas of the discrete case we will prove in Lemma \ref{lemaestabilidad} that every point of a 
positive expansive flow $\phi$ is stable for the inverse flow $\phi^{-1}$. 
The sketch of the proof, for flows without singularities, is the following. 
By contradiction suppose that $x$ is not stable for $\phi^{-1}$. 
So, there is $\epsilon>0$ and $y_j\to x$ 
such that for every reparameterization $h\colon\R\to\R$ and for all $j\geq 0$ there is $t_j\geq 0$ with the property 
$$\dist(\phi^{-1}_{t_j}(y_j),\phi^{-1}_{h(t_j)}(x))=\epsilon$$ 
and 
$$\dist(\phi^{-1}_{t}(y_j),\phi^{-1}_{h(t)}(x))<\epsilon$$
for all $t\in [0,t_j]$.
Now one must notice that a reprameterization may be too \emph{fast} or too \emph{slow}, allowing a \emph{kinematic} 
separation of the trajectories. 
So we will consider a reparameterization $h_j$
that keep the trajectories at a distance smaller than $\epsilon$ for all $t$ in a \emph{maximal} interval $[0,t_j]$.
Assuming that $a_j=\phi^{-1}_{h_j(t)}(x)\to x_ *$ and $b_j=\phi^{-1}_{t_j}(y_j)\to y_*$
we have that $\dist(y_*,x_*)=\epsilon$ and then $x_*\neq y_*$. 
Now what we can prove about this two points is that there is a reparameterization 
$h^*$ such that $\dist(\phi^{-1}_{h^*(t)}x_*,\phi^{-1}_ty_*)\leq \epsilon'$ for all $t \geq 0$, 
being $\epsilon'$ a bit greater than $\epsilon$ but smaller than the expansive constant.
According to the definition of positive expansive flow
we have that $x_*$ and $y_*$ are in a small 
orbit segment. 
Now the maximality of $t_j$ will be contradicted as follows. 
Consider a flow box around the orbit segment containing $x_*$ and $y_*$ as in Figure \ref{flowbox}. 
\begin{figure}[htbp]
\begin{center}   
    \includegraphics{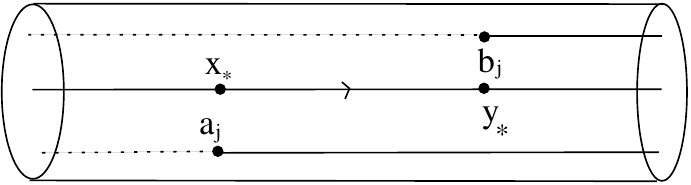}
    \caption{Flow box.}
    \label{flowbox}
\end{center}
\end{figure}

We will show in Section \ref{secthausdorff} that, eventually changing the metric to an equivalent one, 
we have that $\dist(a_j,\phi_{-t} b_j)\leq \dist(a_j,b_j)$ for all $t\geq 0$ suficiently small.
Notice that this is not true in general: 
consider in the Euclidean plane a trajectory like the graph of the function $f(x)=x\sin(1/x)$ with $f(0)=0$.
For vector fields on manifolds a Riemannian metric is enough. 

In Section \ref{sectexpflow} we consider \emph{reparameterizations with rests}, i.e. continuous and surjective maps $h\colon\R\to\R$ such that $h(s)\leq h(t)$ if $s<t$. 
We show that expansiveness and stability can be redefined using this kind of reparameterizations. 
See Propositions \ref{posexpequiv1} and \ref{posexpequiv2} and Remark \ref{stabilityequiv}.
That allows us to get a contradiction because now one can extend a bit the (supposed) maximal time $t_j$ keeping the other point in rest. 
This concludes the sketch of the proof of the stability. 
Similar techniques allow us to prove that the stability is asymptotic and uniform (Lemma \ref{uniform}).

In the discrete case, once one proves the stability, there are different continuations of the proof. 
Our strategy for the continuous case is to prove that periodic orbits do exist, Lemma \ref{periodicas}.  
Then we prove that every orbit is periodic as follows. 
By contradiction suppose that $x\in X$ is not periodic. 
In the $\omega$-limit set of $x$ there is a periodic orbit $\gamma$ because $\omega(x)$ is a compact invariant positive expansive set.
But this contradicts the past asymptotic stability of $\gamma$. 
So, every orbit is periodic, and using the asymptotic stability, we have that 
the number of periodic orbits is finite. 
This concludes the sketch of the proof.

If the flow has singular points the proof is reduced easily to the regular case, this is done in Section \ref{secsing}.

In \cite{Pa} expansive flows on manifolds are studied. There it is shown that every point has a non-trivial stable set. 
The stable set is defined using reparameterizations. So, it implies that positive expansive flows do not exist on compact manifolds 
of dimension greater than one. Their techniques seems to be adaptable for locally connected metric spaces.

\section{Hausdorff distance for a flow}\label{secthausdorff}
In that Section we consider a continuous flow on a compact metric space. 
We construct a metric that is equivalent with the original one and it has \emph{good} properties relative to the flow. 

Let $(X,\dist)$ be a compact metric space. 
Consider $\K$ the set of compact subsets of $X$ equipped with the Hausdorff distance defined by 
\[
 \dist_H(K_1,K_2)=\inf\{\epsilon>0: K_1\subset B_\epsilon(K_2)\hbox{ and } K_2\subset B_\epsilon(K_1)\}.
\]
where $K_1$ and $K_2$ are compact subsets of $X$ and $B_\epsilon(K)=\cup_{p\in K} B_\epsilon(p)$.
It is known that $(\K,\dist_H)$ is a compact metric space. 

Let $\phi\colon X\times\R\to X$ be a continuous flow. 
Denote by $I$ the real interval $[-1,1]$ and
for any $\tau>0$ define $I\tau=[-\tau,\tau]$. Consider the map 
$\phi_{I\tau}\colon X\to\K$ that associates to each point its $I\tau$-orbit segment 
\[
 \phi_{I\tau}(x)=\{\phi_tx:|t|\leq\tau\}.
\]

\begin{prop}
 For every $\tau>0$ the map $\phi_{I\tau}$ is uniformly continuous.
\end{prop}
\begin{proof} By the uniform continuity of the flow on compact intervals of time, 
 we have that given $\epsilon>0$ there is $\delta>0$ such that if $\dist(x,y)<\delta$ 
then $\dist(\phi_tx,\phi_ty)<\epsilon$ for every $t\in I\tau$ and every $x,y\in X$. 
So $d_H(\phi_{I\tau}(x),\phi_{I\tau}(y))<\epsilon$ if $\dist(x,y)<\delta$.
\end{proof}

Notice that if the flow has periodic orbits with arbitrary small periods then $\phi_{I\tau}$ can not be injective. 
We do not consider singularities (i.e. equilibrium points) as periodic points. 
\begin{prop}
 The map $\phi_{I\tau}$ is injective if there are not periodic orbits of period smaller or equal than $3\tau$.
\end{prop}

\begin{proof}
 Arguing by contradiction assume that $\phi_{I\tau}(x)=\phi_{I\tau}(y)$ with $x\neq y$. It implies that $x$ is not singular.
Without loss of generality we can assume that
there is $s\in (0,\tau]$ such that $y=\phi_s(x)$. Then $\phi_{[s-\tau,s+\tau]}x=\phi_{[-\tau,\tau]}x$. 
So, $\phi_{s+\tau} x=\phi_{s'} x$ for some $s'\in I\tau$. Therefore $\phi_{s+\tau-s'}x=x$.
This is a contradiction because $0<s+\tau-s'\leq 3\tau$ and $x$ is not singular.
\end{proof}

Notice that expansive flows (with or without singular points) and flows without singular points (expansive or not) 
have not arbitrary small periods.

Assuming that $\phi_{I\tau}$ is injective we consider the following distance in $X$
\[
 \tilde \dist(x,y)=\dist_H(\phi_{I\tau}(x),\phi_{I\tau}(y)).
\]
\begin{prop}
If $\phi_{I\tau}$ is injective then the new distance $\tdist$ is equivalent with $\dist$.
\end{prop}

\begin{proof}
 Since $\phi_{I\tau}$ is continuous and $X$ is compact, the image of $\phi_{I\tau}$ is compact. 
So $\phi_{I\tau}\colon X\to\phi_{I\tau}(X)$ is an open map and the inverse $\phi_{I\tau}^{-1}\colon \phi_{I\tau}(X)\to X$ is continuous. 
Then $(X,\dist)$ and $(\phi_{I\tau}(X),\dist_H)$ are homeomorphic. 
The distance $\tdist$ in $X$ is the pull-back of $\dist_H$ by $\phi_{I\tau}$, so 
$\dist$ and $\tdist$ are equivalent metrics in $X$.
\end{proof}

The following Propositions deals with the question that we will state now. Consider a flow box $U$ centered at $x$. Take $y$ close to $\phi_{-t_0}x$ for some $t_0>0$. 
Is it true that $\tdist(y,x)\geq \tdist(\phi_ty,x)$ for small and positive values of $t$? 
According to the arguments that we will do in the next Sections, it is enough to answer these questions for flows without singular points. 
To continue we need the following Lemma. It is stated for the inverse flow, defined as $\phi^{-1}_t=\phi_{-t}$, because in the following 
Proposition it will be used in this way.

\begin{lemma}\label{lemax}
If $\phi^{-1}_sx\neq x$ for all $x\in X$ and $s\in (0,3\tau]$ then there is $\tilde\tau>0$ such that for all $p\in X$, 
$\dist(p,\phi^{-1}_tp)<\dist(p,\phi^{-1}_{\theta+2\tau}p)$ 
for all $p\in X$ and $\theta\in [0,\tilde\tau]$.
\end{lemma}

\begin{proof}
By contradiction assume that there is $\theta_n>0$, $\theta_n\to 0$, and $p_n\to p_*$ 
such that $\dist(p_n,\phi^{-1}_{\theta_n}p_n)\geq \dist(p_n,\phi^{-1}_{\theta_n+2\tau}p_n)$ for all $n\geq 0$. 
Then, in the limit, we have the contradiction $\phi^{-1}_{2\tau}p_*=p_*$.
\end{proof}

Now we can prove the main result of the section. We assume that there are no periods smaller than $3\tau$.

\begin{prop}\label{propmonotonialocal}
 If $\phi$ has not singular points then for all $t_0\in (0,\tilde\tau]$ there is $\delta>0$ and $t_1\in(0,t_0)$ such that if 
$\dist(\phi_{t_o}y,x)<\delta$ and $0\leq s\leq u\leq t_1$ then $\tdist(\phi_s y,x)\geq\tdist(\phi_uy,x)$.
\end{prop}

\begin{proof}
 By contradiction assume that there is $t_0\in (0,\tilde\tau]$, sequences $x_n,y_n\in X$ and $s_n,u_n\in\R$ 
such that $\phi_{t_0}y_n\to z$, $x_n\to z$, $0\leq s_n\leq u_n\to 0$ and
\begin{equation}\label{abs}
\tdist(\phi_{s_n}y_n,x_n)<\tdist(\phi_{u_n}y_n,x_n)
\end{equation}
for all $n\geq 0$. Inequality (\ref{abs}) means that there is $\epsilon_n>0$ such that 
\begin{itemize}
 \item[(a)] $\phi_{I\tau}(\phi_{s_n}y_n)\subset B_{\epsilon_n}(\phi_{I\tau}(x_n))$ and 
\item[(b)] $\phi_{I\tau}(x_n)\subset B_{\epsilon_n}(\phi_{I\tau}(\phi_{s_n}y_n))$ 
\end{itemize}
but 
\begin{itemize}
 \item [(c)] $\phi_{I\tau}(\phi_{u_n}y_n)\nsubseteq B_{\epsilon_n}(\phi_{I\tau}(x_n))$ or
\item [(d)] $\phi_{I\tau}(x_n)\nsubseteq B_{\epsilon_n}(\phi_{I\tau}(\phi_{u_n}y_n))$.
\end{itemize}

In that paragraph we show that $\epsilon_n$ does no converge to 0. 
By (a) we have that there is $w_n\in I\tau$ such that 
\begin{equation}\label{ineq*}
\dist(\phi_{-\tau+s_n}y_n,\phi_{w_n}x_n)<\epsilon_n.
\end{equation} 
Taking a subsequence we can assume that $w_n\to w_*\in I\tau$. 
Taking limit in the inequality (\ref{ineq*}) and supposing that $\epsilon_n\to 0$ we have that $\phi_{-\tau-t_0}z=\phi_{w_*}z$. 
This is a contradiction because $z=\phi_{\tau+t_0+w_*}z$ and $|\tau+t_0+w_*|<3\tau$. 
So, taking a subsequence of $\epsilon_n$, we assume that $\epsilon_n\to\epsilon_*>0$.

Assume that (c) holds. It implies that there is $v_n\in I\tau$ such that for all $t\in I\tau$ 
\begin{equation}\label{ineq1'}
 \dist(\phi_{v_n+u_n}y_n,\phi_tx_n)\geq\epsilon_n.
\end{equation}
Now we show that $v_n\to \tau$. By (a) we have that for all $s\in I\tau$, there is $t\in I\tau$ such that 
\begin{equation}\label{ineqa'}
 \dist(\phi_{s+s_n}y_n,\phi_tx_n)<\epsilon_n.
\end{equation}
Using the inequalities (\ref{ineq1'}) and (\ref{ineqa'}) we have that $s+s_n\neq v_n+u_n$ for all $s\in I\tau$. 
But $v_n\in I\tau$, so $v_n\in (\tau -(u_n-s_n),\tau]$. Then $v_n\to\tau$. 

Now, taking limit in the inequality (\ref{ineq1'}) we have that $\dist(\phi_{\tau-t_0}z,\phi_tz)\geq\epsilon_*$ for all $t\in I\tau$. 
So we can put $t=\tau-t_0$ and $\dist(z,z)\geq\epsilon_*>0$ which is a contradiction. Then (c) can not hold.

Now assume that (d) is true. Condition (d) means that there is $v_n\in I\tau$ such that for all $t\in I\tau$ we have 
\begin{equation}\label{ineq2'}
 \dist(\phi_{v_n}x_n,\phi_{t+u_n}y_n)\geq\epsilon_n.
\end{equation}
By (b) we have that there is $w_n\in I\tau$ such that 
\begin{equation}\label{ineq3}
 \dist(\phi_{v_n}x_n,\phi_{s_n+w_n}y_n)<\epsilon_n.
\end{equation}

We will show that $w_n\to -\tau$. By (\ref{ineq2'}) and (\ref{ineq3}) we have that $s_n+w_n\neq t+u_n$ for all $t\in I\tau$. 
Then $w_n\notin [-\tau+u_n-s_n,\tau+u_n-s_n]$ but $w_n\in I\tau$. Therefore $w_n\in [-\tau, -\tau +u_n-s_n)$ and $w_n\to -\tau$. 

Assuming that $v_n\to v_*\in I\tau$ and taking limit in (\ref{ineq2'}) we have that 
\begin{equation}\label{ineq4} 
\dist(\phi_{v_*}z,\phi_{t-t_0}z)\geq \epsilon_* 
\end{equation}
for all $t\in I\tau$. Also, taking limit in (\ref{ineq3}) we have 
\begin{equation}\label{ineq6} 
 \dist(\phi_{v_*}z,\phi_{-\tau-t_o}z)\leq \epsilon_*.
\end{equation}
By (\ref{ineq4}) and the fact that $\epsilon_*>0$ we have that $v_*\neq t-t_0$ for all $t\in I\tau$. 
Then $v_*\in (\tau-t_0,\tau]$. 
If $t=\tau$ in inequality (\ref{ineq4}) we have that 
$\dist(\phi_{v_*}z,\phi_{\tau-t_0}z)\geq \epsilon_*.$
This and inequality (\ref{ineq6}) contradicts Lemma \ref{lemax}, with $\theta=v_*-(\tau-t_0)$ and $p=\phi_{v_*}z$, because $t_0\in (0,\tilde\tau]$.
\end{proof}

\begin{prop}\label{propmonotonia}
For all $t_2\in(0,\tilde\tau]$ there is $\delta>0$ and $t_1>0$ such that if 
$\tdist(\phi_tx,y)<\delta$ or $\tdist(x,\phi_{-t}y)<\delta$
for some $t\in [t_2,\tilde\tau]$ and $0\leq s\leq u\leq t_1$ then $\tdist(\phi_sy,x)\geq\tdist(\phi_uy,x)$.
\end{prop}

\begin{proof}
 It follows by Proposition \ref{propmonotonialocal} and the compactness of the interval $[t_2,\tilde\tau]$.
\end{proof}

\section{Expansive flows}\label{sectexpflow}

In that section we present the definition of expansive flow and some useful equivalences. 
We state them for positive expansiveness but they have their counterpart for expansive flows.
We consider flows without singular points. In Section \ref{secsing} we consider the singular case.

Let $\rep^+$ be the set of all increasing homeomorphisms $h\colon\R\to\R$ such that $h(0)=0$. Such maps are called \emph{reparameterizations}.

\begin{df}\label{BWexp}
A continuous flow $\phi$ on a compact metric space $X$ is \emph{positive expansive} if for every $\epsilon>0$ there is $\delta>0$ such that 
if $\dist(\phi_{h(t)}x,\phi_ty)<\delta$ for all $t\geq 0$, with $x,y\in X$ and $h\in\rep^+$, 
then $y\in\phi_{I\tau}x$. 
\end{df}

Recall that $y\in\phi_{I\tau}x$ if and only if there is $t\in I\tau=[-\tau,\tau]$ such that $y=\phi_tx$.
This is the \emph{positive} adaptation of the definition given by R. Bowen and P. Walters in \cite{BW}.
Now we present an equivalent definition. 
Consider $\rep$ as the set of non-decreasing, surjective and continuous maps $h\colon \R\to\R$ such that $h(0)=0$. 
By \emph{non-decreasing} we mean: if $s<t$ then $h(s)\leq h(t)$. 
The idea is to allow a point to stop the clock for a while (recall that in \cite{Ma} reparameterizations are called \emph{clocks}).
The maps of $\rep$ will be called \emph{reparameterizations with rests}.

Define the set of pairs of reparameterizations with rests 
$$\rep^2=\{g=(h_1,h_2):h_1,h_2\in \rep\}$$ 
and extend de action of $\phi$ to $X\times X$ as $\phi_t(x,y)=(\phi_tx,\phi_ty)$.
Also we define 
$$\phi_{g(t)}(x,y)=(\phi_{h_1(t)}x,\phi_{h_2(t)}y)$$ 
for $g=(h_1,h_2)\in\rep^2$.
We now consider the Fréchet distance defined by
\[
 \dist_F(x,y)=\inf_{g\in\rep^2}\sup_{t\geq 0}\dist(\phi_{g(t)}(x,y)).
\]
This distance was introduced in \cite{Fre} in the begining of the Theory of metric spaces. 
It was first defined for compact curves but, as noticed in \cite{Ma}, it can be extended to non-compact 
trajectories. 


\begin{prop}\label{posexpequiv1}
A flow $\phi$ is positive expansive if and only if
for all $\epsilon>0$ there is $\delta>0$ such that if $\dist_F(x,y)<\delta$ 
then $x$ and $y$ are in an $\epsilon$-orbit segment.
\end{prop}

\begin{proof}
 The converse follows because $id_\R\in\reparam\subset\reparamx$. 
The direct part is a consequence of the following Lemma.
\end{proof}

\begin{lemma}\label{lemareparamx}
For all $\delta>0$ there is $\delta'>0$ such that if 
$\dist_F(x,y)<\delta'$ then there is $h\in\reparam$ such that 
$\dist(\phi_{h(t)}x,\phi_ty)<\delta$ for all $t\geq 0$.
\end{lemma}

\begin{proof}
Consider $\delta'\in (0,\delta)$ and $\gamma>0$ such that $\dist(x,\phi_tx)<(\delta-\delta')/2$ for all $x\in X$ 
and for all $t\in (-\gamma,\gamma)$. 
Take two increasing sequences $s_n$ and $t_n$ such that $h_x(t_n)=h_y(s_n)=n\gamma$ for all $n\geq 1$, starting with $s_0=t_0=0$.
Then define $h_1(t_n)=h_2(s_n)=n\gamma$ and extend piecewise linearly. 
In this way we have that $|h_1(t)-h_x(t)|,|h_2(t)-h_y(t)|<\gamma$ for all $t\geq 0$. 
Then by the triangular inequality it follows that $h=h_1\circ h_2^{-1}$ works.
\end{proof}

Consider the set $T_\epsilon(x,y)\subset \R^2$ of pairs of positive numbers $(t_x,t_y)$ such that 
there is $g\in\rep^2$ and $s>0$ such that $\dist(\phi_{g(t)}(x,y))\leq\epsilon$ 
for all $t\in[0,s]$ and $g(s)=(t_x,t_y)$.
In $\R^2$ we consider the norm $\|(a,b)\|=|a|+|b|$ (the properties of that specific norm will be used in the next Section). 
\begin{obs}
 If $T_\delta(x,y)$ is not bounded then $\pi_1 T_\delta(x,y)$ and $\pi_2 T_\delta(x,y)$ are not bounded, where $\pi_i(x_1,x_2)=x_i$, $i=1,2$, are the canonical projections of $\R^2$.
\end{obs}

\begin{lemma}\label{lemareparamx'}
For all $\delta'>0$ there is $\delta>0$ such that if $\dist(\phi_{g(t)}(x,y))<\delta$ for all $t\in[0,T]$ 
and some $g\in\rep^2$
then there is $h\in\reparam$ such that $\dist(\phi_{h(t)}x,\phi_ty)<\delta'$ for all $t\in[0,h(T)]$.
\end{lemma}
\begin{proof}
 Use the same technique of Lemma \ref{lemareparamx}.
\end{proof}

If $\dist_F(x,y)<\epsilon$ then $T_\epsilon(x,y)$ is not bounded, as can be seen from the definitions. 
The following Proposition is a kind of converse. Its proof is based on the proof of Lemma 9 in \cite{Th}.

\begin{prop}\label{thomas}
For all $\epsilon>0$ there is $\delta>0$ such that if $T_\delta(x,y)$ is not bounded then $\dist_F(x,y)<\epsilon$. 
\end{prop}

\begin{proof}
For $\epsilon>0$ given consider $\gamma>0$ such that 
\begin{equation}\label{asterisco1}
\hbox{if $\dist(x,y)<\epsilon/2$ and $|t|<\gamma$ then 
$\dist(\phi_tx,y)<\epsilon$.} 
\end{equation}
Take $\delta'\in(0,\epsilon/2)$ such that 
\begin{equation}\label{asterisco2}
\hbox{if $\dist(x,y)<\delta'$ then $\dist(\phi_{\pm\gamma} x,y)>\delta'$.} 
\end{equation}
Finally, pick $\delta>0$ from Lemma \ref{lemareparamx'} associated to $\delta'$. We will show that this value of $\delta$ works.
Suppose that for some $x,y\in X$ we have that $T_\delta(x,y)$ is not bounded. 
So, for all $n\geq 1$ there are $h'_x,h'_y\in\reparamx$ and $T>0$ 
such that $$\dist(\phi_{h'_x(t)}x,\phi_{h'_y(t)}y)<\delta$$ for all $t\in[0,T]$ and $h'_y(T)=n$.
Then by Lemma \ref{lemareparamx'} there is $h^n_x\in\rep$ such that 
$$\dist(\phi_{h^n_x(t)}x,\phi_ty)<\delta'$$
for all $t\in[0,n]$.
Eventually taking a subsequence we can suppose that there is an increasing sequence $w_n\to\infty$ such that 
$h_x^n(w_n)=n\gamma$ and 
$$\dist(\phi_{h^n_x(t)}x,\phi_ty)<\delta'$$
for all $t\in[0,w_n]$. We will define $h\in\reparamx$ such that $$\dist(\phi_{h(t)}x,\phi_ty)<\epsilon$$ 
for all $t\geq 0$. 
Define $h(w_n)=h_x^{n}(w_n)=n\gamma$ for all $n\geq 0$. For $t\in[0,w_1]$ define $h(t)=h_x^1(t)$. 
Now consider $t\in (w_{n-1},w_n)$. 
To define $h(t)$ we consider two cases.
\begin{enumerate}
 \item If $h_x^{n-1}(w_{n-1})\leq h_x^n(w_{n-1})$ then $h(w_{n-1})=h_x^{n-1}(w_{n-1})$ and extend linearly for $t\in (w_{n-1},w_n)$.
\item If $h_x^{n-1}(w_{n-1})> h_x^n(w_{n-1})$ consider $z\in (w_{n-1},w_n)$ such that $h_x^n(z)=(n-1)\gamma$. 
Define $h(t)=(n-1)\gamma$ for all 
$t\in[w_{n-1},z]$ and extend linearly for $t\in [z,w_n]$.
\end{enumerate}
By condition (\ref{asterisco2}) we have that $|h(t)-h_x^n(t)|\leq \gamma$ for all $t\in [w_{n-1},w_n]$ and $n\geq 1$. 
Then, since $\dist(\phi_{h_x^n(t)}x,\phi_ty)<\delta'<\epsilon/2$, we have by condition (\ref{asterisco1}) that
$$\dist(\phi_{h(t)}x,\phi_ty)<\epsilon$$
for all $t\geq 0$ and the proof ends.
\end{proof}

Here is another characterization of expansiveness that will be useful.

\begin{prop}\label{posexpequiv2}
 A flow $\phi$ is positive expansive if and only if for all $\epsilon>0$ there is $\delta>0$ such that if $T_\delta(x,y)$ is not bounded then $x$ and $y$ 
are in a $\epsilon$-orbit segment.
\end{prop}

\begin{proof}
 Suppose that $\phi$ is positive expansive. Consider $\epsilon>0$ given. 
By Proposition \ref{posexpequiv1} there is $\delta'$ such that if $\dist_F(x,y)<\delta'$ then they are in a $\epsilon$-orbit segment. 
Now take from Proposition \ref{thomas} a positive $\delta$ such that if $T_\delta(x,y)$ is not bounded then 
$\dist_F(x,y)<\delta'$.
This finishes the direct part. 

The converse follows because if $\dist_F(x,y)<\delta$ then $T_\delta(x,y)$ is not bounded.
\end{proof}

\section{Stability}\label{sectstability}

In that Section we assume that the flow has not singular points. 
We introduce the concept of Lyapunov stability allowing reparameterizations of the trajectories. The stability properties of positive expansive flows are stated.
We assume that the metric of the space is $\tdist$, defined in Section \ref{secthausdorff}, but we will denote it simply as $\dist$. 

We start defining Lyapunov stability according to the Fréchet distance as was done in \cites{Ma,Pa}.

\begin{df}\label{dfstable}
We say that $x$ is \emph{stable} if for every $\epsilon>0$ there is $\delta>0$ 
such that if $\dist(x,y)<\delta$ then $\dist_F(x,y)<\epsilon$,
i.e. there is 
a pair of reparameterizations with rests $g\in\rep^2$ such that $\dist(\phi_{g(t)}(x,y))<\epsilon$ for all $t\geq 0$.
\end{df}

\begin{obs}\label{stabilityequiv}
 By Lemma \ref{lemareparamx} we have that $x$ is stable if and only if for every $\epsilon>0$ there is $\delta>0$ 
such that if $\dist(x,y)<\delta$ then 
there is a 
reparameterization $h\in\rep^+$ such that $\dist(\phi_tx,\phi_{h(t)}y)<\epsilon$ for all $t\geq 0$.
\end{obs}

\begin{df} 
We say that $(T_x, T_y)$ in the closure of $T_\epsilon(x,y)$ is a \emph{maximal pair of times} 
for $(\epsilon,x,y)$ if for all $(t_x,t_y)\in T_\epsilon(x,y)$ 
we have that $\|(T_x,T_y)\|\geq \|(t_x,t_y)\|$ for the sum norm in $\R^2$. 
\end{df}

In the following result we use the properties of $\tdist$. 
For this we will consider the positive number $\tilde\tau$ given in Lemma \ref{lemax} and the interval $I{\tilde\tau}=[-\tilde\tau,\tilde\tau]$.
As usual, we define the distance between a point $a\in X$ and a set $A\subset X$ as 
$\dist(a,A)=\inf\{\dist(a,x):x\in A\}$.

\begin{prop}\label{propmaxpair}
For all $\epsilon>0$ there is $\sigma>0$ such that if $(T_x,T_y)$ is a maximal pair of times for $(\epsilon,x,y)$ 
then 

\begin{center}
$\dist(\phi_{T_x}x,\phi_{I{\tilde\tau}}(\phi_{T_y}y))>\sigma$ and 
$\dist(\phi_{T_y}y,\phi_{I{\tilde\tau}}(\phi_{T_x}x))>\sigma$.
\end{center}
\end{prop}

\begin{proof} 
Given $\epsilon>0$ consider $t_2>0$ such that $\phi_{[-t_2,t_2]}x\subset B_\epsilon (x)$ for all $x\in X$.
For this value of $t_2$ take $\delta>0$ and $t_1>0$ from Proposition \ref{propmonotonia}. 
Consider $\sigma\in(0,\delta)$ such that 
\begin{equation}\label{ecu1}
\hbox{if } y\notin B_\epsilon(x) \hbox{ then } \dist(\phi_{[-t_2,t_2]}x,y)>\sigma. 
\end{equation}
Notice that $\dist(\phi_{T_x}x,\phi_{T_y}y)=\epsilon$. 
By contradiction assume that 
$$\dist(\phi_{T_y}y,\phi_{I{\tilde\tau}}(\phi_{T_x}x))\leq\sigma,$$ 
being the other case symetric. 
By condition (\ref{ecu1}) there is $t_0\in [-\tilde\tau,-t_2]\cup [t_2,\tilde\tau]$ such that $$\dist(\phi_{T_y}y,\phi_{t_0}\phi_{T_x}x)\leq\sigma.$$
Suppose that $t_0\in [t_2,\tilde\tau]$ (the other case is similar).
Now take $g\in\rep^2$, $(T'_x,T'_y)\in\R^2$ and $s>0$ such that 
$
 \dist(\phi_{g(t)}(x,y))<\epsilon
$
for all $t\in[0,s]$, 
\begin{equation}\label{ecu2}
\|(T'_x,T'_y)-(T_x,T_y)\|< t_1 
\end{equation}
and $g(s)=(T'_x,T'_y)$. 
We define $\hat g\in\rep^2$ as 
\[
\hat g(t)=
\left\{
\begin{array}{ll}
 g(t) & \hbox{for all } t\leq s,\\
 g(s)+(t-s,0) & \hbox{if } t\in[s,s+t_1],\\
 g(s)+(t-s,t-s-t_1) & \hbox{if } t\geq s+t_1.  
\end{array}
\right.
\]
So, for $t\in[s,s+t_1]$ we have, by Proposition \ref{propmonotonia}, 
that $\dist(\phi_{\hat g(t)}(x,y))\leq\dist(\phi_{\hat g(s)}(x,y))<\epsilon$. 
Then $g(s+t_1)=(T'_x+t_1,T'_y)\in T_\epsilon(x,y)$ and by inequality (\ref{ecu2}) we have that $\|g(s+t_1\|>\|(T_x,T_y)\|$ contradicting 
the maximality of $(T_x,T_y)$. 
\end{proof}

Given $\epsilon>0$ and $x,y\in X$ we consider the following set of pairs of reparameterizations with rests
\[
 \rep^2_\epsilon(x,y)=\{g\in\rep^2:\dist(\phi^{-1}_{g(t)}(x,y))<\epsilon \hbox{ for all } t\geq 0\}.
\]
The following result says that if two points are close enough then $\rep^2_\epsilon(x,y)$ is not empty if 
$\phi$ is positive expansive without singular points.
Notice that positive expansiveness do not depend on the metric (defining the same topology). 
Therefore we will assume that $\dist$ has the properties of the metric $\tdist$ defined in Section \ref{secthausdorff}.
\begin{lemma}\label{lemaestabilidad}
 If $\phi$ is positive expansive then every point is stable for $\phi^{-1}$ with uniform $\delta$.
\end{lemma}

\begin{proof}
By Proposition \ref{posexpequiv2} there is an expansive constant
$\epsilon'>0$ such that if $T_{\epsilon'}(x,y)$ is not bounded then $y\in \phi_{I{\tilde\tau}}x$.
By contradiction assume that there is $\epsilon\in(0,\epsilon')$ and two sequences $x_j,y_j$ such that 
$T_\epsilon(x_j,y_j)$ is bounded for all $j\in\N$.
For each $j$ consider $(T_{x_j}, T_{y_j})$ a maximal pair of times for $(\epsilon,x_j,y_j)$ associated to $\phi^{-1}$.
By the continuity of the flow we have that $T_{x_j},T_{y_j}\to\infty$ as $j\to\infty$. 
Eventually taking subsequences, we can assume that $\phi_{T_{x_j}}x_j\to x_*$ and $\phi_{T_{y_j}}y_j\to y_*$. 
By Proposition \ref{propmaxpair} we have that $x_*$ and $y_*$ are not in a $\tilde\tau$-orbit segment.
Also, for every $T>0$ we have that there is $g\in\rep^2$ and $s>0$ such that 
$\dist(\phi_{g(t)}(x_*,y_*))< \epsilon'$ for all $t\in[0,s]$ and $\|g(s)\|\geq T$. 
So, $T_{\epsilon'}(x_*,y_*)$ is not bounded and 
it contradicts the positive expansiveness of the flow (as stated in Proposition \ref{posexpequiv2}) because $x_*$ and $y_*$ are not in a $\tilde\tau$-orbit segment.
\end{proof}

The following Lemma states the uniform asymptotic stability for $t\to -\infty$.

\begin{lemma}\label{uniform}
 If $\phi$ is positive expansive then for all $\epsilon>0$ there is $\delta>0$ such that for all $\sigma>0$ 
there is $T>0$ such that if $\dist(x,y)<\delta$ then there is $g\in\rep^2_\epsilon(x,y)$ such that 
$\dist(\phi^{-1}_{g(t)}(x,y))<\sigma$ if $\|g(t)\|\geq T$.
\end{lemma}
\begin{proof}
Given $\epsilon>0$ smaller than an expansive constant, consider $\delta>0$ from Lemma \ref{lemaestabilidad}. 
By contradiction we will show that this value of $\delta$ works.
So, suppose that there is $\sigma>0$, $T_n\to\infty$ and
$x_n,y_n\in X$ such that $\dist(x_n,y_n)<\delta$ and 
\begin{equation}\label{lema3.6*1}
\begin{array}{l}
\hbox{for all $g\in\rep^2_\epsilon(x_n,y_n)$ there is $t\geq 0$ such that}\\
\hbox{$\|g(t)\|\geq T_n$ and $\dist(\phi^{-1}_{g(t)}(x_n,y_n))\geq\sigma$.} 
\end{array}
\end{equation}
Again by Lemma \ref{lemaestabilidad} there is $\delta'$ such that 
\begin{equation}\label{3.6**}
 \hbox{if $\dist(u,v)<\delta'$ then 
$\rep^2_\sigma(u,v)$ is not empty.}
\end{equation}

For each $n$ take $g_n\in\rep^2_\epsilon(x_n,y_n)$ and
consider $t_n$ such that $\|g_n(t_n)\|= T_n-\tilde\tau$. 
Let $(u_n,v_n)=\phi^{-1}_{g_n(t_n)}(x_n,y_n)$.
By conditions (\ref{lema3.6*1}) and (\ref{3.6**}) there is $\delta''$ such that 
$\dist(\phi_tu_n,v_n)\geq\delta''$ and 
$\dist(u_n,\phi_tv_n)\geq\delta''$ if $|t|\leq \tilde\tau$.
So, limit points of $u_n$ and $v_n$ are not in a $\tilde\tau$-orbit segment and contradict positive expansiveness.
\end{proof}

\section{Positive expansiveness}\label{sectposexpflow}

In that Section we prove the main result of the article for flows without singular points.
First we show that positive expansive flows has periodic orbits. 
The idea to find such trajectories is to show that there is a compact invariant set that is 
a suspension and apply the result for positive expansive homeomorphisms.

\begin{lemma}\label{periodicas}
Every positive expansive flow has at least one periodic orbit.
\end{lemma}

\begin{proof}
Consider $\epsilon'>0$ such that for all $y\in X$
\begin{equation}\label{eqper1}
 \hbox{if } (\tilde h,\tilde h')\in \rep^2_{2\epsilon'}(y,y) \hbox{ then } |\tilde h(t)-\tilde h'(t)|<\tilde\tau/2\hbox{ for all } t\geq 0.
\end{equation}
This condition will be used bellow to show that the map $f$ is well defined. 
Take a recurrent point $x$ and $t_n\to +\infty$ such that $\phi^{-1}_{t_n}(x)\to x$.
For any $\epsilon\in(0,\epsilon')$ consider $\delta>0$ from Lemma \ref{uniform}. 
Let $S\subset B_\delta(x)$ be a compact local cross section of time $\tilde\tau$, $x\in S$, 
and consider the flow box $U=\phi_{[-\tilde\tau,\tilde\tau]}(S)$. 
Consider $r>0$ such that
\begin{equation}\label{eqper2}
 \phi_{[-\tilde\tau/2,\tilde\tau/2]}B_r(x)\subset U.
\end{equation}
For $\sigma=r/2$ in Lemma \ref{uniform} take the corresponding $T>0$. 
Let $N>0$ be such that $\dist(\phi^{-1}_{t_N}x,x)<r/2$ and $t_N>T$. 
By Lemma \ref{uniform}, for all $y\in S$ ($S\subset B_\delta(x)$) there is $g\in\rep^2_\epsilon(x,y)$ such that:
$$\dist(\phi^{-1}_{g(t)}(x,y))<\sigma=r/2$$
if $\|g(t)\|\geq T$. If $g=(h_x,h_y)$ there is $s\geq 0$ such that $h_x(s)=t_N$. 
Then $\|g(s)\|\geq T$ and $\phi^{-1}_{h_y(s)}y\in B_r(x)\subset U$. 
Consider $\pi\colon U\to S$ the projection on the flow box.
Let $f\colon S\to S$ be defined by 
\[
 f(y)=\pi(\phi^{-1}_{h_2(s)}y)
\]
if $s\geq 0$ and $g=(h_1,h_2)\in\rep^2_\epsilon(x,y)$ satisfies:
\begin{enumerate}
\item $h_1(s)=t_N$ and
\item $\phi^{-1}_{h_2(s)}y\in B_r(x)$.
\end{enumerate}
We have shown that for all $y\in S$ there are $s$ and $g$ satisfying this conditions. 

In this paragraph we will show that $f$ is well defined, 
i.e. do not depend on $g$ and $s$. 
Consider $s,s'\geq 0$ and $g=(h_1,h_2),g'=(h_1',h_2')\in\rep^2_\epsilon(x,y)$ 
satisfying both items above. 
Recall that $\epsilon'>\epsilon$ and consider two increasing reparameterizations 
$\hat h_1$ and $\hat h'_1$ such that 
\begin{itemize}
 \item $\dist(\phi^{-1}_{\hat h_1(t)}x,\phi^{-1}_{h_2(t)}y)<\epsilon'$ for all $t\geq0$,
 \item $\dist(\phi^{-1}_{\hat h'_1(t)}x,\phi^{-1}_{h'_2(t)}y)<\epsilon'$ for all $t\geq0$ and
 \item $\hat h_1(s)=t_N=\hat h'_1(s')$.
\end{itemize}
So, if we define 
$(\tilde h,\tilde h')=(h_2\circ \hat h_1^{-1},h'_2\circ \hat h'^{-1}_1)$ we have that 
\begin{itemize}
 \item $\dist(\phi^{-1}_{t}x,\phi^{-1}_{\tilde h(t)}y)<\epsilon'$ for all $t\geq0$,
 \item $\dist(\phi^{-1}_{t}x,\phi^{-1}_{\tilde h'(t)}y)<\epsilon'$ for all $t\geq0$,
 \item $h_2(s)=\tilde h(t_N)$ and $h'_2(s')=\tilde h'(t_N)$.
\end{itemize}
and by the triangular inequality 
\[
 \dist(\phi^{-1}_{\tilde h(t)}y,\phi^{-1}_{\tilde h'(t)}y)<2\epsilon'
\]
for all $t\geq 0$. Then by condition (\ref{eqper1}) we have that 
$$|h_2(s)-h'_2(s)|=|\tilde h(t_N)-\tilde h'(t_N)|<\tilde\tau/2.$$
This inequality joint with equation (\ref{eqper2}) and 
the fact that $\phi^{-1}_{h_2(s)}y,\phi^{-1}_{h'_2(s')}y\in B_r(x)$ implies that 
the points $\phi^{-1}_{h_2(s)}y$ and $\phi^{-1}_{h'_2(s')}y$ are in the same orbit segment contained in the flow box $U$. 
So, they have the same projection in section $S$ and $f$ is well defined.

Now we will show that $f$ is continuous. 
Given $y\in S$ consider $s\geq 0$ and $g=(h_1,h_2)\in\rep^2_\epsilon(x,y)$ satisfying the definition 
of $f(y)$. 
Consider $\rho>0$ such that for all $y'\in B_\rho (y)\cap S$ 
we have that $\phi^{-1}_{h_2(s)}y'\in B_r(x)$. 
Then the continuity of $f$ follows by the continuity of the flow $\phi$ and the continuity of the projection $\pi$.

Now one can restrict $f$ to the compact invariant set $$K=\cap_{n\geq 0} f^n(S)$$ and 
notice that $f$ is a negative expansive homeomorphisms on $K$ because $\phi$ is positive expansive in $\phi_\R(K)$. 
We conclude that $K$ is finite and $f$ has periodic points. So $\phi$ has periodic orbits.
\end{proof}

\begin{teo}\label{teo1}
 If $\phi$ is a positive expansive flow without singular points then $X$ is the union of a finite number of periodic orbits.
\end{teo}

\begin{proof}
 First we show that every orbit is periodic. 
By contradiction assume that there is a point whose orbit is non-compact. 
By Lemma \ref{periodicas} there is a periodic orbit contained in $\omega(x)$. 
But it contradicts Lemma \ref{uniform}. 
Again by Lemma \ref{uniform} there is just a finite number of periodic orbits and the proof ends.
\end{proof}

\section{Singular flows}\label{secsing}

Now we consider positive expansive flows with singular points. 
A change in the definition is needed because singularities are isolated points of the space if the flow is expansive according to Definition \ref{BWexp}
(even if one consider expansiveness instead of positive expansiveness). 
So, for singular flows we consider the following definition. 

\begin{df}\label{komuroexp}
 A continuous flow $\phi$ in a compact metric space $X$ is \emph{positive expansive} if for all $\epsilon>0$  
there is $\delta>0$ such that 
if $\dist(\phi_{h(t)}x,\phi_ty)<\delta$ for all $t\geq 0$, with $x,y\in X$ and $h\in\reparam$, 
then $x$ and $y$ are in an orbit segment of diameter smaller than $\epsilon$.
\end{df}

This is the \emph{positive} adaptation of the definition given in \cite{Ar} for expansive flows with singular points.
Definitions \ref{BWexp} and \ref{komuroexp} coincide if the flow has not singular points.
\begin{teo}\label{teo2}
If $\phi$ is a positive expansive flow with singular points then $X$ is the union of finite periodic orbits and singularities.
\end{teo}

\begin{proof}
Let $\epsilon>0$ be an expansive constant. 
We will show that singularities are stable for $\phi^{-1}$. 
By contradiction assume there is $x_n\to p$, $x_n\neq p$, $p$ a singular point, and 
for all $n\in\N$ there is $t_n\geq 0$ such that $\dist(\phi^{-1}_{t_n}x_n,p)=\epsilon$. 
If $y_n=\phi^{-1}_{t_n}x_n$ converges to $q$, then $q\neq p$ and $\phi_tq\to p$ as $t\to\infty$. 
So, $p$ and $q$ contradict the positive expansiveness of the flow. 
Therefore there is $\delta>0$ such that if $\dist(x,p)<\delta$ then $\phi^{-1}_tx\in B_\delta(p)$ for all $t\geq 0$. 
We will show that $B_\delta(p)=\{p\}$. By contradiction suppose there is $\dist(x,p)\in (0,\delta)$. 
By hypothesis there is $t>$ such that $\phi_tx\notin B_\epsilon(p)$. So $x$ is not periodic. 
By the stability of singularities there is no singular point in $\omega(x)$. Then $\omega(x)$ is positive expansive, connected and free of singularities. 
By Theorem \ref{teo1} it is a periodic orbit. But this contradicts the stability of periodic orbits, i.e. Lemma \ref{lemaestabilidad}.
So, singular points are isolated points of $X$ and the proof is reduced to Theorem \ref{teo1}.
\end{proof}


\end{document}